\documentclass[11pt]{amsart}

\usepackage{amssymb}
\usepackage{amsmath}
\usepackage{amsthm}
\usepackage{enumerate}
\usepackage{graphicx}
\usepackage{rotating}
\usepackage{amsrefs}
\usepackage{comment}
\usepackage{nicefrac}
\usepackage{caption}

\usepackage{xcolor}

\usepackage{tikz}
\usepackage[T1]{fontenc}
\usetikzlibrary{patterns}
\tikzset{
	convexset/.style = {line width = 0.5 pt, fill=lightgray, opacity=0.3},
	ext/.style = {circle, inner sep=0pt, minimum size=2pt, fill=black},
	segment/.style = {line width = 0.75 pt}
}

\theoremstyle{plain}

\newtheorem{theorem}{Theorem}[section]

\newtheorem{proposition}[theorem]{Proposition}
\newtheorem{corollary}[theorem]{Corollary}

\theoremstyle{definition}

\newtheorem{definition}[theorem]{Definition}
\newtheorem{remark}[theorem]{Remark}
\newtheorem{example}[theorem]{Example}

\newcommand{\dist}{d}

\newcommand{\R}{\mathbb{R}}
\newcommand{\N}{\mathbb{N}}

\newcommand{\maxi}{\mathrm{Max}}
\newcommand{\stmax}{\mathrm{StMax}}
\newcommand{\pos}{\mathrm{Pos}}
\newcommand{\inn}{\mathrm{inn}\,}

\renewcommand{\epsilon}{\varepsilon}
\renewcommand{\phi}{\varphi}

\author[A.~Daniilidis]{Aris Daniilidis}

\address{Institute of Statistics and Mathematical Methods in Economics, E105-04, TU Wien, Wiedner Hauptstra{\ss }e 8, A-1040 Wien}

\email{aris.daniilidis@tuwien.ac.at}

\author[C.A.~De~Bernardi]{Carlo Alberto De Bernardi}

\address{Dipartimento di Matematica per le Scienze economiche, finanziarie ed attuariali, Universit\`a Cattolica del Sacro Cuore, Via Necchi 9, 20123 Milano} 

\email{carloalberto.debernardi@unicatt.it} \email{carloalberto.debernardi@gmail.com}

\author{Enrico Miglierina}
\address{Dipartimento di Matematica per le Scienze economiche, finanziarie ed attuariali, Universit\`{a} Cattolica del Sacro Cuore, Via Necchi 9, 20123 Milano.} 

\email{enrico.miglierina@unicatt.it}

\title[ABB theorems:  results and limitations in infinite dimensions]{ABB theorems:  results and limitations\medskip \\ in infinite dimensions}



\begin{document}

\begin{abstract} We construct a weakly compact convex subset of $\ell^{2}$ with non\-empty interior that has an isolated maximal element, with respect to the lattice order~$\ell _{+}^{2}$. Moreover, the maximal point cannot be supported by any strictly positive functional, showing that the Arrow-Barankin-Blackwell theorem fails. This example discloses the pertinence of the assumption that the cone has a bounded base for the validity of the result in infinite dimensions. Under this latter assumption, the equivalence of the notions of strict maximality and maximality is established.
\end{abstract}

\maketitle

\section{Introduction}

Let $X$ be a Banach space and $P$  a closed convex cone with base
$\mathcal{B}$, that is, $\mathcal{B}$ is a closed convex subset of $P$ with
the property that every nonzero element $x\in P$ can be represented in a
unique way in the form $x=\lambda b,$ with $\lambda>0$ and $b\in\mathcal{B}$.
Notice that this is equivalent to the existence of a functional $f\in X^{\ast}$ that takes strictly positive values on $P\setminus\{0\}$. In what follows we
denote the set of strictly positive functionals on $P$ by 
$$\mathrm{inn}P^{\ast}=\bigl\{f\in X^*;\, f(u)>0\ \forall u\in P\setminus\{0\} \bigr\}.$$
Nonemptiness of the set $\mathrm{inn}P^{\ast}$ clearly implies (and in finite
dimensions is equivalent to) the fact that $P$ is a pointed cone, that is,
$P\cap(-P)=\emptyset$. Moreover, nonemptiness of the interior of $\mathrm{inn}P^*$ is equivalent to the existence of a bounded base for the cone~$P$.\smallskip\newline 
The cone $P$ defines a partial order on $X$ as follows: $x\succeq y\Longleftrightarrow x-y\in P$. In case
$X=\mathbb{R}^{n}$ and $P=\mathbb{R}_{+}^{n}$, the famous
Arrow-Barankin-Blackwell theorem (in short, ABB theorem) asserts that every
maximal element of a compact convex set can be approximated by positive elements, that is, points that support the set by means of a strictly positive
functional, see~\cites{ABB53, Peleg72}. The result has a pertinent economic interpretation: every optimal allocation of commodities can be approximated by allocations that are supported by a nontrivial system of prices (\cites{ABB53, Econometrica}).
Because of its importance, a lot of effort has been devoted to extensions of the ABB result in infinite dimensional spaces. The most general result can be announced in locally convex topological vector spaces and ensures the density of the positive elements in the set of maximal elements of every convex compact set (Theorem~\ref{Thm-ABB}). Therefore, if $X$ is a Banach space equipped with a cone $P$ and $K$ is a compact (respectively, weakly compact) convex subset of~$X$, then every maximal element of $K$ can be strongly (respectively, weakly) approximated by a sequence of positive elements. Moreover, if the cone~$P$ has a bounded base, then the approximation is always strong, even if~$K$ is merely weakly compact. The same conclusion obviously holds under the assumption that the maximal element to be approximated is a point of continuity of the identity map of $K$ from the weak to the norm topology (Corollary~\ref{cor:PC}). Notwithstanding, until now, the degree of necessity of these assumptions was not made sufficiently clear. \smallskip\newline 
In this article we construct an example showing that strong approximation fails in general even for the classical separable Hilbert space $\ell^{2}$ equipped with its lattice cone~$\ell_{+}^{2}$ (Proposition~\ref{example: max isolated}). Our counterexample concerns a weakly compact convex set with nonempty interior, framework in which the Hahn-Banach separation theorem applies, outlining that the problem stems from the fact that the lattice cone~$\ell_{+}^{2}$ does not have a bounded base. Indeed, assuming that $P$ has a bounded base would guarantee that a strong approximation ABB result holds for any weakly compact set. The current work shows that this assumption (or some variant of it) is essentially necessary.\smallskip\newline
Let us quote some equivalent forms of the assumption that the cone has a bounded base, that have already been employed in the literature, see e.g. \cites{Jahn88,Petschke90, GalSal93}: the cone~$P$ is of Bishop-Phelps type, the dual cone $P^{\ast}$ has nonempty interior, there exists a functional that strongly exposes $0$ in $P$ and finally that~$0$ is a point of continuity of the identity map of the cone~$P$ from the weak to the norm topology. We refer to~\cite{D2000} for a detailed discussion of the assumptions. We shall also show that under any of these assumptions, we can reinforce the notion of maximality: every maximal element is also strictly maximal, notion that relates to stability (see~\cite{CasMig10} and references therin).

\section{Notation and preliminaries}
Throughout this paper $E$ stands for a locally convex topological vector space, while $X$ denotes a Banach space with dual space $X^*$. The closed unit  ball, the   open unit ball, and the unit sphere of $X$ are denoted by $B_X$, $U_X$, and $S_X$, respectively. We also write $\delta_n\searrow 0^+$ to denote that $\delta_n >0$ and $\lim_{n\to\infty}\,\delta_n=0$.\smallskip\newline
We denote by
\begin{equation}
\mathrm{Max}(K,P)=\bigl\{x\in K:\,\{x\}=K\cap(x+P)\bigr\}\label{eq:max}
\end{equation}
the set of $P$-maximal elements of a nonempty subset $K$ of $X$, and by
\begin{equation}
\mathrm{Pos}(K,P)=\bigl\{x\in K:\exists f\in\mathrm{inn}\,P^{\ast},f(x)=\sup
f(K)\bigr\}\label{eq:pos}
\end{equation}
the set of its positive elements. \smallskip\newline

Let us mention that if $K$ is a (weakly) compact convex set, 
$\mathrm{Max}(K,P)$ is nonempty, as a consequence of Zorn lemma and a standard compactness argument. Nonemptiness of $\mathrm{Pos}
(K,P)$ is less obvious and will follow from forthcoming Theorem~\ref{Thm-ABB}.  \smallskip
\newline In 1953, Arrow-Barankin-Blackwell, in \cite{ABB53}, established the density of the set
$\mathrm{Pos}(K,P)$ in $\mathrm{Max}(K,P),$ provided $K$ is a compact convex
subset of $\mathbb{R}^{n}$ and $P=\mathbb{R}_{+}^{n}$.
Since then this density result became relevant in Economic Theory, see \cite{Econometrica} for updated references. 
General ABB results have also been obtained in arbitrary Banach
spaces (see, e.g., \cites{Jahn88,Petschke90})
and later on in locally convex topological vector spaces  (\cites{Fu96,GalSal93}). The following theorem summarizes the previous results. The proof, which is essentially already known, is based on the notion of dilating cones (see \cite{BorZhu93}). For the convenience of the reader we provide a sketch of it.

\begin{theorem}
[Abstract density result]\label{Thm-ABB} Let $(E,\Im)$ be a locally convex
topological vector space, $K$ a $\Im$-compact convex subset of $E$, and $P$ a
closed convex cone with  base $\mathcal{B}$. Then
\begin{equation}
\mathrm{Pos}(K,P)\subseteq\mathrm{Max}(K,P)\subseteq\overline{\mathrm{Pos}
(K,P)}^{\Im}\label{eq:density}
\end{equation}

\end{theorem}

\begin{proof} (Sketch) We provide a sketch of proof for the special
case in which $E=X$ is a Banach space and $\Im$ is either the weak or the norm
topology of $X,$ which is actually the case that is relevant in this work. The
arguments can be easily adapted to cover the more abstract setting.\smallskip \newline Assume that $K$ is $\Im$-compact and convex, $P=\overline
{\mathrm{cone}}(\mathcal{B})$ and $\bar{x}\in\mathrm{Max}(K,P)$, that is, $\{\bar{x}\}=K\cap(\bar{x}+P)$. Then we consider the closed, convex cone
\[
P_{n}=\overline{\mathrm{cone}}(\mathcal{B}+\delta_n B_X),\qquad\text{where }\, \delta_n\searrow 0^+.
\]
It follows that for $n$ sufficiently large, $P_{n}$ has a base and $P=\bigcap_{n\geq
1}P_{n}$.\smallskip\newline
For each $n\geq1$, choose a $P_{n}$-maximal point $x_{n}
\in\mathrm{Max}(K,P_{n})$, such that
\[
x_{n}\in K_{n}:=K\,\bigcap\,(\bar{x}+P_{n}).
\]
Notice that $P=\bigcap_{n\geq1}P_{n}$ readily yields $\{\bar{x}\}=\bigcap_{n\geq 1}K_{n}$. Therefore, by the $\Im$-sequentially compactness of $K$ (if $\Im$ is the weak topology use Eberlein-\v Smulian theorem),  we easily obtain
\[
\Im\text{-}\underset{n\rightarrow\infty}{\lim}x_{n}=\bar{x}.
\]
Since $\{x_{n}\}=K\cap(x_{n}+P_{n})$ and $\mathrm{int}
(P_{n})\neq\emptyset$, there exists a functional $x^{\ast}\in P_{n}^{\ast}$ that
supports the set $K$ at the point $x_{n}$. Since $x^{\ast}$ is actually a
strictly positive functional for the original cone $P,$ the proof is complete.
\end{proof}
\medskip

\begin{remark}
 Let us mention that the authors in~\cite{GTZ2004}, working with a refinement of the notion of $P$-maximality (Henig proper maximality), were able to obtain a result in the spirit of Theorem~\ref{Thm-ABB} under weaker assumptions: nonconvex sets $K$ have been considered and the assumption of $\mathcal{F}$-compactness was replaced by $\mathcal{F}$-asymptotic compactness, see \cite[Theorem~4.1]{GTZ2004}.    
\end{remark}

In a Banach space $X$, Theorem~\ref{Thm-ABB} expresses simultaneously two different density results,
for the norm and respectively, for the weak topology. However, assuming norm compactness is very restrictive in infinite dimensions, while on the other hand, concluding only weak approximation is suboptimal. Therefore, it is desirable to obtain a strong approximation result for weakly compact sets. However, this was achieved only under additional assumptions. Jahn~\cite{Jahn88} was the
first to derive a norm approximation result for weakly compact subsets, by assuming that the cone $P$ was of ``Bishop-Phelps type''. Subsequently Petschke~\cite{Petschke90} (see also~\cite{GalSal93} for a different approach) refined Jahn's proof to conclude the same result for cones $P$ having a bounded base. More recent related references include \cites{GTZ2004, GMP21, HJ23}. \smallskip\newline 
We resume these results in a general scheme presented below. To this end, let us recall the following definition.

\begin{definition}
[Point of continuity]\label{pcon}Let $A$ be a nonempty subset of a Banach space $X.$ We say that $\bar{x}\in A$ is a \textit{point of continuity} for the set $A$ and we denote $\bar{x}\in\mathrm{PC}\bigl(A,(w,\Vert\cdot\Vert)\bigr)$, if the identity mapping $$\mathrm{Id}:(A,w)\rightarrow(A,\Vert\cdot\Vert)$$ is continuous at
$\bar{x}$.
\end{definition}

It is well-known (see, e.g., \cite{D2000}) that a closed convex
pointed cone $P$ has a bounded base if and only if $0\in\mathrm{PC}
\bigl(A,(w,\Vert\cdot\Vert)\bigr).$ In view of the above, a careful
investigation of the proof of Theorem~\ref{Thm-ABB} leads readily to the
following corollary (see also \cite{BedSon99}*{Theorem~3.1}).

\begin{corollary}
[Density result with combined topologies]\label{cor:PC}Let $K$ be a nonempty $w$-compact
convex subset of $X$, $P$ a convex closed cone with base and let $\overline{x}
\in\mathrm{Max}(K,P)$. Then $\overline{x}\in\overline{\mathrm{Pos}
(K,P)}^{\parallel\cdot\Vert}$ provided that one of the following conditions is satisfied:\smallskip
\begin{enumerate}
\item $0\in\mathrm{PC}\bigl(P,(w,\|\cdot\|)\bigr)$ (equivalently, $P$ has a bounded base); \smallskip
\item $\overline{x}\in\mathrm{PC}\bigl(K,(w,\Vert\cdot\Vert)\bigr)$.
\end{enumerate}
\end{corollary}

\medskip

\section{Main results}
\noindent The main results are twofold. In the first subsection  we consider the same framework as in Corollary~\ref{cor:PC} (where the strong density result holds) and show that in this case, every $\bar{x}\in \mathrm{Max}(K,P)$ is also a strictly maximal element, see forthcoming definition in (\ref{eq:stmax}). This latter is  a more restrictive notion of maximality introduced to study stability and well-posedness in vector optimization (see \cites{Bednar07, CasMig10} and references therein).\smallskip\newline
In the second subsection we show that the ABB density result fails for general weakly compact convex sets in a separable Hilbert space where the ordering cone $P$ is the natural lattice cone. The reason is the lack of bounded base for this cone.
\medskip
 \subsection{Relation between $\mathrm{Max}(K,P)$ and $\mathrm{StMax}(K,P)$}
 Let us start with the definition of strict maximality
(\cites{Bednar07,CasMig10}): \smallskip
 \begin{equation}\label{eq:stmax}
\mathrm{StMax}(K,P)=\bigl\{x\in K:\,\forall\varepsilon>0,\exists
\delta>0,(P+\delta B_{X})\cap(K-x)\subset\varepsilon B_{X}
\bigr\}.
\end{equation}

\medskip

\begin{figure}
	\centering 
		\begin{tikzpicture}[scale=0.8]
			\draw[convexset]
	(0,0) .. controls + (0.6,1) and + (-1,-0.4) .. (2,2)  .. controls + (1,0.4) and + (-0.4,0.4) ..(5.69,1.69) .. controls + (0.4,-0.4) and + (0.2,1) .. (6,0) .. controls + (-0.2,-1) and +(-0.4,-4) ..(0,-2) .. controls + (0,1) and +  (0,-1) ..(0,0) -- cycle;
     
		\draw (5.69,1.69) circle[radius=1pt];
		\fill (5.69,1.69)  circle[radius=1pt];
		\node at (5.4,1.6) {\small{$0$}};
		\draw[-,dashed] (5.69,1.69) circle[radius=1.61];	
		
        \draw[dashed, line width = 1.5 pt] (5.19,1.19) arc (-135:-45:0.707);        
		\draw[-,thick, line width = 0.5 pt](5.69,1.69)--(9.69,5.69);
        \draw[-,dashed, line width = 1.5 pt](6.19,1.19)--(10.69,5.69);	
		\draw[-,thick, line width = 0.5 pt](5.69,1.69)--(1.69,5.69);
        \draw[-,dashed, line width = 1.5 pt](5.19,1.19)--(0.69,5.69);
        
		\fill[fill=gray, opacity=0.5] (5.69,1.69)--(9.69,5.69)--(1.69,5.69);
		\pattern[pattern=north east lines, opacity=0.5] (5.19,1.19)--(6.19,1.19)--(10.69,5.69)--(0.69,5.69);
        \pattern[pattern=north east lines, opacity=0.5] (5.19,1.19) arc (-135:-45:0.707);

        \draw[-,dashed, line width = 0.5] ((5.69,1.69)--(5.69,3.3);
        \node at (6,2.6) {{$\varepsilon$}};
        \draw[-,dashed, line width = 0.5] ((8,4)--(8.5,3.5);
        \node at (8.5,4) {{$\delta$}};
			
		\node at (5.69,4.5) {{$P$}};
		\node at (1.3,4) {{$P+\delta B_{X}$}};	
		\node at (2,-.4) {{$K-\bar{x}$}};
  	\end{tikzpicture}	
		\caption {Definition of strict maximality: $\bar{x} \in \mathrm{StMax}(K,P)$}  \label{FIG-1}
\end{figure}
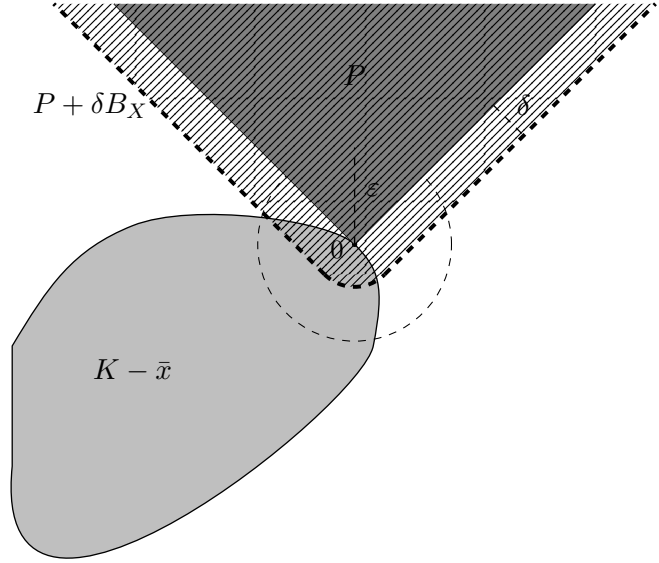

\noindent The above definition is illustrated by Figure~\ref{FIG-1}. It is easy to see that every {strictly maximal point is  maximal}, but the converse is not true (see forthcoming Proposition~\ref{example: stmax empty} for example). However, under the assumption of Corollary~\ref{cor:PC}, we will show that $\mathrm{Max}(K,P)$ and $\mathrm{StMax}(K,P)$ coincide and they are both nonempty. To do so, we need to recall the following result stated in a general locally convex space, see~\cite{Bednar07}*{Theorem~2.2.1}
\begin{theorem}[Strict maximality in locally convex spaces] 
\label{th: B} Let $E$ be locally convex topological vector space, $K\subset E$ compact convex, and $P$ a closed convex cone with base. Suppose that $x\in \mathrm{Max}(K,P)$. Then for each neighbourhood~$\mathcal{W}$ of the origin there exist a neighbourhood $\mathcal{V}$ of the origin such that $$(P+\mathcal{V})\cap(K-x)\subset \mathcal{W}.$$ 
\end{theorem}
Notice that for the special case in which $E=X$ is a Banach space considered with its norm topology, we deduce from Theorem~\ref{th: B} that $\mathrm{Max}(K,P)$ coincides with $\stmax(K,P)$, whenever $K$ is (norm) compact. This is of course a very restrictive assumption in infinite dimensions. The following result remedies partially this inconvenience.

\begin{theorem}\label{th: stmax coincide with max} Let $K$ be a $w$-compact convex subset of $X$, $P$ a closed convex cone with base, and $\overline x\in \maxi(K,P)$. Then $\overline x\in \stmax(K,P)$ provided one of the following conditions is satisfied:
\begin{enumerate}
    \item $0\in\mathrm{PC}\bigl(P,(w,\|\cdot\|)\bigr)$  (equivalently, $P$ has a bounded base);
    \item $\overline{x}\in\mathrm{PC}\bigl(K,(w,\|\cdot\|)\bigr)$.
\end{enumerate}
\end{theorem}

\begin{proof} 
We may  suppose without any loss of generality that $\overline{x}=0$. We first prove the conclusion in the case in which  (i) holds. Let $\epsilon >0$ and take a base $\mathcal{B}$ of $P$ contained in $\frac{\epsilon}{3} B_X$. By the Hahn-Banach theorem, there exists a functional $f\in S_{X^*}$ such that 
$$0=f(0)\leq\sup f(K)<\alpha\leq\inf f(\mathcal{B})\qquad \text{(see  Figure~\ref{FIG-2}).}$$

\begin{figure}
	\centering 
	
	\begin{tikzpicture}[scale=0.8]
	
		\draw[convexset]
	(0,0) .. controls + (0.6,1) and + (-1,-0.4) .. (2,2)  .. controls + (1,0.4) and + (-0.4,0.4) ..(5.69,1.69) .. controls + (0.4,-0.4) and + (0.2,1) .. (6,0) .. controls + (-0.2,-1) and +(-0.4,-4) ..(0,-2) .. controls + (0,1) and +  (0,-1) ..(0,0) -- cycle;

		\draw[-](0.5,3.7)--(9.6,.5);
		\draw (5.69,1.69) circle[radius=1pt];
		\fill (5.69,1.69)  circle[radius=1pt];
		\node at (5.05,1.6) {\small{$\overline x=0$}};
		\draw[-,dashed] (5.69,1.69) circle[radius=3];	
		
		\draw[line width = 1.2 pt](3.69,3.69)--(6.09,2.09);
	
		\draw[-,thick, line width = 0.5 pt](5.69,1.69)--(9.69,5.69);
		\draw[-,thick, line width = 0.5 pt](5.69,1.69)--(1.69,5.69);
		\fill[fill=gray, opacity=0.5] (5.69,1.69)--(9.69,5.69)--(1.69,5.69);

		\node at (5.69,5.3) {{$\overline x+P$}};
		\node at (5.2,3) {{$\mathcal B$}};	
		\node at (7.2,-.2) {{$\frac13\epsilon B_X$}};	
		\node at (1.2,2.9) {\small{$\{f=\alpha\}$}};
	
		\node at (2,-.4) {{$K$}};			
	\end{tikzpicture}
 
\caption{Proof of Theorem~\ref{th: stmax coincide with max}(i)}\label{FIG-2}
\end{figure}
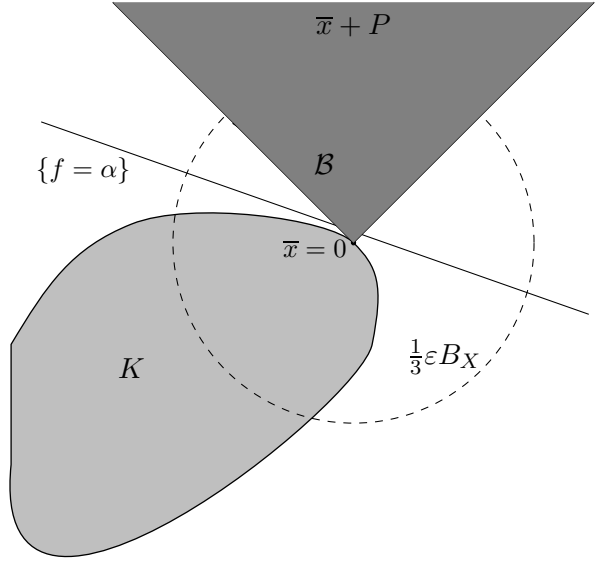

Let $0<\delta\leq\min\{\,\alpha,\, \epsilon/3\,\}$ and let us show that $(P+\delta B_X)\cap K\subset \epsilon B_X$. To do this, let $y=p+\delta b\in K$ with $p\in P$ and $b\in B_X$ and observe that $$f(p)=f(y-\delta b) <\alpha+\delta\leq2\alpha.$$ 
It follows that $f(p/2)\leq \alpha$, therefore $\frac{p}{2}\in\frac{\varepsilon}{3}B_X$ and consequently $$\|y\|\leq \|p\|+\delta\leq \epsilon.$$ The conclusion follows.

Let us now suppose that (ii) holds and let us consider $\epsilon>0$. Since $$0\in\mathrm{PC}\bigl(K,(w,\|\cdot\|)\bigr),$$ there exists a $w$-neighbourhood $\mathcal{W}$ of the origin such that 
$$0\in\mathcal{W}\cap K\subset \epsilon B_X.$$
By Theorem~\ref{th: B}, there exists a $w$-neighbourhood (in particular, a norm
neighbourhood) of the origin $\mathcal{V}$ such that
$$
(P+\mathcal{V})\cap K\subset \mathcal{W}\cap K\subset\varepsilon B_{X},
$$
and the proof is complete.
\end{proof}
\bigskip

\noindent The following proposition shows that neither of the assumptions~(i), (ii) in Theorem~\ref{th: stmax coincide with max} can be removed.

\begin{proposition}\label{example: stmax empty} Let $X=\ell^2$. Define 
$$ P=\{x=(x_n)\in X;\,  nx_1-|x_n|\geq 0,\  \forall n\geq 2\}\quad \text{and}\quad K=-P\cap B_X.$$
Then $$\maxi (K,P)=\pos(K,P)=\{0\}\quad \text{and}\quad \stmax(K,P)=\emptyset.$$
    
\end{proposition}

\begin{proof}
Since the cone is pointed, we have $P\cap K\subset P\cap (-P)=\{0\}$ and consequently $0\in\maxi(K,P)$. We claim that the origin is the unique element of the set $\maxi(K,P)$. Indeed, if $\overline{x}\in K\setminus \{0\}$ then $\overline{x}\in -P$ and consequently $0\in\overline{x}+P$, proving the claim. Moreover, it is easily seen that $e_1\in\inn P^*$ and  $\sup e_1(K)=0$. Therefore we deduce that $0\in\pos(K,P)$. \smallskip\newline Let us now prove that $0\not\in \stmax(K,P)$. To do this, let $n\in\N$ and define the sequence $$z^n=(z^n_k)_{k}\in \ell^2$$ by
        $$\textstyle z^n_k=\begin{cases}

-\frac{1}{2 (n+1)}\,, &\text{if}\  k=1 \smallskip \\
\phantom{--}{\frac{1}{2}}\,, &\text{if}\  k= n+1 \smallskip \\
\phantom{--}0\,, &\text{if}\  k\notin \{1, n+1\}
        \end{cases}.$$
Then it is easy to see that $\{z^n\}_n\subset K$. For every $n\geq 1$, let us set 
\[
w^n_1=-z^n_1=\frac{1}{2(n+1)} \quad \text{and}\quad w^n_k=z^n_k,\,\text{ for }\, k\neq 1.
\]
Then $\{w^n\}_n \subset P$ and 
$\dist(z^n, P)\leq \|z^n-w^n\|=\frac{1}{n+1} \to 0$ (as $n\to\infty$)
whereas $$\|z^n\|^2=\frac{1}{4(n+1)^2}+\frac{1}{4}\geq\frac14,\quad \text{for all }\, n\geq 1.$$ The claim follows.
\end{proof}
\bigskip

\noindent We shall now present an example showing that in Theorem~\ref{th: stmax coincide with max},  weak compactness of the set $K$ cannot be replaced by the weaker assumption that the set $K$  is weakly closed and bounded. It is a slight modification of  \cite{CasMig10}*{Example~4.6}.

\begin{example}\label{ex:bounded in ell_1}
Let $X=\ell^1$ equipped with its lattice cone $P=\ell^1_+$. It is worth pointing out that $P$ has both bounded and unbounded bases. {Indeed, let us consider a strictly positive functional $f=(f_n)_{n\geq 1}\in \ell^\infty_{++}\equiv \mathrm{inn}(\ell^{1}_{+})^{\ast}$. If there exists $\alpha >0$ such that $f_n\geq \alpha$ for every $n\geq 1$, then $f^{-1}(1)\cap P$ is a bounded base for $\ell^1_+$, otherwise $f$ generates an unbounded base.} In particular, taking a functional of the latter type:
\[
f:=(1/n)_{n\geq 1} \in \mathrm{inn}(\ell^1_+)^{\ast}
\]
the set $f^{-1}(1)\cap P$ is an unbounded base for $P$. We fix this cone and consider the closed convex bounded set 
\[
K=\{x\in \ell^1: -1\leq f(x) \leq 0\} \,\bigcap\, 2B_X.
\]
It follows directly that $0\in \mathrm{Pos}(K,P)$ and \textit{a fortiori} $0\in \mathrm{Max}(K,P)$. \smallskip\newline
On the other hand, taking $x^n := e_n - \frac{1}{n} e_1\in \frac{1}{n}B_X+P$, for all $n\in\N$, where $\{e_n \}_n$ is the standard unit-vector basis of $\ell^1$, we deduce that $\{x^n\}_n \subset K$ and $\|x^n \|\geq 1$ for every $n\in\N$. This shows that $0\notin \mathrm{StMax}(K,P)$.    
\end{example}
\medskip

\noindent The following proposition shows that if the cone $P$ has a bounded base (as was the case in the previous example), the set $\mathrm{StMax}(K,P)$ is always nonempty. Therefore, in Example~\ref{ex:bounded in ell_1}, the set of strict maxima is nonempty (and strictly contained in $\mathrm{Max}(K,P)$). This being said, in forthcoming Proposition~\ref{prop: renorming of c_0} we shall see that if $P$ does not have a bounded base, the set of strict maxima can be empty. \smallskip\newline
We recall that a functional $f\in X^*\setminus \{0\}$ is called a supporting functional of $K$ at $x_0$ if $f(x_0)=\sup f(K)$.

\begin{proposition}[existence of strict maxima for a bounded based cone]\label{th: stmax nonempty} Let $K$ be a nonempty closed convex bounded subset of a Banach space~$X$ and $P$ a convex closed cone with a bounded base. Then $\stmax(K,P)\neq \emptyset$.
\end{proposition}
\begin{proof} Since $P$ has a bounded base, then
$\mathrm{int}P^{\ast}$ is nonempty. Hence, by Bishop-Phelps theorem (see, e.g., \cite{FHHMPZ2001}*{Theorem 7.41}), there exists $x^*_0\in\mathrm{int}P^{\ast}$ that is a supporting functional for $K$ at a point $x_0\in K$, i.e., $x_0 \in \mathrm{Pos}(K,P)$. By  \cite{CasMig10}*{Corollary~5.4}, we conclude that $x_0\in \stmax(K,P)$.
\end{proof}
\medskip

\noindent Let us further denote by $c_0:=\{x=(x_n)_n\in\R^{\N}:\, \underset{n\to \infty}{\lim} x_n =0 \}$ the Banach space of all null sequences from $\N$ to $\R$ equipped with the norm $\|x\|_{\infty}:=\underset{n\geq 1}{\sup}\,|x_n|$. Let us consider the linear operator $T:c_0\rightarrow \ell^2$ defined for every $x\in c_0$ as follows:  
\begin{equation}\label{def-op-T}
x:=(x_n) \,\mapsto \, T(x):= \left(\frac{x_n}{2^n}\right)_{n\geq1}\in\ell^{2}.
\end{equation}
Notice that $T$ is injective, continuous and 
\[
\|Tx\|_2= \sqrt{\sum_{n\geq 1}\frac{|x_n|^2}{4^n}} \, \leq \, \left(1/\sqrt{3}\right) \, \|x\|_{\infty}\,.
\]
Therefore, 
\begin{equation}\label{eq:equiv-norm}
|||x||| \,:=\,\|x\|_{\infty} \, + \, \|Tx\|_{2}
\end{equation}
is an equivalent norm on $c_0$. \medskip\newline
\noindent We are now ready to provide an example where the set 
$\mathrm{StMax}(K,P)$ is empty, even if $\mathrm{Max}(K,P)$ is nonempty, showing the pertinence of the assumption that the cone $P$ has a bounded base in Proposition~\ref{th: stmax nonempty}. 

\begin{proposition}[example where $\mathrm{StMax}(K,P)$ is empty]\label{prop: renorming of c_0}
Consider the Banach space $X=(c_0,|||\cdot|||)$, where $|||\cdot|||$ is the equivalent norm defined in \eqref{eq:equiv-norm}. Consider further the closed convex bounded set $K=B_X$ and the lattice cone
    $$P:=(c_0)_+ =\{ x=(x_n)\in c_0:\, x_n\geq 0,\ \text{for all }   n\geq 1 \}.$$
Then 
\[
\mathrm{StMax}(K,P)=\emptyset \qquad \text{and} \qquad \mathrm{Max}(K,P)\neq\emptyset.
\]
 \end{proposition}

\begin{proof}
Let us denote by $\{e_n\}_n$ the canonical basis of $c_0$ and set $\alpha:= 1+1/\sqrt{3}$. It follows that
$$\|x\|_{\infty}\, \leq |||x||| \, \leq \alpha \,\|x\|_{\infty},\quad\text{for all } x\in c_0. $$
Set $\bar x:=\frac{2}{3} e_1$ and notice that for all $p\in P\setminus\{0\}$ we have $|||\bar x||| = 1 < |||\bar x + p |||$.
It follows readily that $\bar x = \frac{2}{3} e_1\in \mathrm{Max}(K,P)$, therefore the set of maxima of $K$ is not empty. \smallskip\newline
\noindent It remains to show that $\mathrm{StMax}(K,P)=\emptyset$. It is sufficient to prove that no point in the boundary $S_X$ of $K$ can be in $\mathrm{StMax}(K,P)$. \smallskip\newline
To this end, let $x=(x_n)\in S_X$ (that is, $|||x|||=1$). Since $\stmax(K,P)\subset\maxi(K,P)$ we can clearly also assume that $x\in\maxi(K,P)$. Therefore 
\[
|||x|||=1 \, < \, |||x+ p|||,\quad \text{ for all }\, p\in P\setminus\{0\}.
\]
Notice further that $\|x\|_{\infty}\geq 1/{\alpha}$. Therefore, since $x\in c_0$, there exists $n_0\in\N$ such that $|x_n|\leq 1/{2\alpha}$, for all $n\geq n_0$. Setting $$p_n:=(1/2{\alpha})\,e_n \in P$$ we have:
\[
\textstyle 
|||x|||\,=\,1 \,< \, |||x+p_n||| :=\, \|x +\frac{1}{2\alpha}e_n\|_{\infty} \, +\,\left(\underset{k\in\N\setminus\{n\}}{\sum} (x_k^2/4^k) \,+\,\frac{1}{4^n}\left(\,x_n+\frac{1}{2\alpha}\right)^2\right)^{1/2}.
\]
Notice that $\|x +\frac{1}{2\alpha}e_n\|_{\infty}=\|x\|_{\infty}$, for $n\geq n_0$ and that $\beta_n:=|||x+\frac{1}{2\alpha}e_n|||\to 1$. Set $$z_n=\frac{1}{\beta_n}(x+p_n)\in B_X\equiv K$$
and notice that $d(z_n,x+P)\,\leq\, |||z_n-(x+p_n)|||\to 0$ (since $\beta_n\to1$). \smallskip\newline
On the other hand, by the triangular inequality we obtain
\[ 
|||(x+p_n)-z_n|||\, +\, |||z_n-x |||\, \geq\, |||p_n||| \, >\, \|p_n\|_{\infty} = \frac{1}{2\alpha}\,.
\]
Since $|||(x+p_n)-z_n||| \to 0$, we deduce that 
$z_n-x \notin (1/2\alpha)U_X$, for all $n$ sufficiently large, whence $x\not\in\mathrm{StMax}(K,P)$. \end{proof}


\subsection{Failure of approximation of $\mathrm{Max}(K,P)$ by $\mathrm{Pos}(K,P)$ }

In this subsection we construct an example showing that the strong version of the ABB theorem fails to hold. This outlines the pertinence of the assumptions of all known infinite dimensional versions of the result, revealing in particular  that the assumptions  in  Corollary~\ref{cor:PC} cannot be removed. \smallskip\newline
Indeed, we construct a weakly compact convex set in $\ell^2$ with an isolated maximal point which is not a point of continuity and cannot be supported by a strictly positive functional with respect to the natural ordering cone $\ell^2_+$. It is worth pointing out that the set $K$ in this example will have nonempty interior. On the other hand, the cone $P$ does not (and cannot) have any bounded base.

\begin{proposition}[failure of ABB strong approximation]\label{example: max isolated}  Let $X=\ell^{2}$, $P=\ell_{+}^{2}$, and
$$ K=\{x=(x_n)\in X;\,  x_1+x_n^2\leq 0,\  \forall n\geq 2\}.$$
Then $$\maxi (K,P)\cap U_X=\{0\}\quad \text{and}\quad \stmax (K,P)\cap U_X=\pos (K,P)\cap U_X=\emptyset.$$
    \end{proposition}
    \begin{proof}
        Let us notice that $$K\cap\{x=(x_n)\in X;\, x_1\geq 0\}=\{0\}.$$
        Indeed, let $x=(x_n)\in K\setminus \{0\}$, then by definition we have $x_1<0$ and consequently $x$ cannot be maximal for $P=\ell_{+}^{2}$. On the other hand, we readily have $0\in\maxi(K,P)$. \smallskip\newline
        We claim that $0\not\in\pos(K,P)$. Indeed, let $y=(y_n)\in\inn P^*=\ell^2_{++}$ be arbitrarily chosen. Then taking $\alpha>0$ sufficiently small, the point $x=(-\alpha, \sqrt{\alpha},0,\ldots)$ belongs to $K$ and $y(x):=-\alpha y_1+\sqrt{\alpha}y_2>0=y(0)$, showing that $y$ cannot support $K$ at $0$. \smallskip\newline
        Let us now claim that $0\not\in\stmax(K,P)$. To prove this, let $n\in\N$ and define $z^n=(z^n_k)_k\in \ell^2$ by
        $$\textstyle z^n_k=\begin{cases}

-\frac{1}{\sqrt2}\,\frac{1}{n}, &\text{if}\  k=1 \smallskip \\
\phantom{-}\frac{1}{\sqrt{2}}\,\frac{1}{\sqrt{n}}, &\text{if}\  k\in\{2,\ldots, n+1\} \smallskip \\
\phantom{--}0, &\text{if}\  k>n+1
        \end{cases}.$$
Then it is easy to see that $\{z^n\}\subset K$. Setting $w^n_k=\max\{z^n_k,0\}$ for all $n,k\geq 1$, we have $\{w^n\}_n\subset P$ and $\dist(z^n, P)\leq \|z^n-w^n\| \to 0$ (as $n\to\infty$), whereas 
\[
\|z^n\|^2=\frac{1}{2n^2}+n\frac{1}{2n}\geq\frac{1}{2},\quad\text{for } \,n\ge 1.
\]
Therefore the claim follows.\smallskip\newline
        Let us prove that $\maxi (K,P)\cap U_X=\{0\}$. To do this, let $x=(x_n)\in K $ be such that $x_1<0$ and $x\in U_X$. Then the set 
        $$\mathcal{N}_x:=\{n\in\N;\, n\geq 2,\, x_n^2=|x_1|  \}$$
        is finite (and possibly empty). Take $n_0>1$ such that $n_0\not\in\mathcal{N}_x$. Then for $\epsilon>0$ sufficiently small we have $|x_1|>(|x_{n_0}|+\epsilon)^2$. It follows that $$J:=[x-\epsilon e_{n_0},x+\epsilon e_{n_0}]\subset K,$$
        and consequently $x\not\in\maxi (K,P)$, since necessarily $x+P$ intersects $J\setminus\{x\}$.
        The fact that $$\stmax (K,P)\cap U_X=\pos (K,P)\cap U_X=\emptyset$$ follows directly from the inclusion $$\stmax (K,P)\cup \pos (K,P)\subset \maxi (K,P).$$
The proof is complete.    \end{proof}

\bigskip

\begin{remark} It is well-known that the closed unit ball $B$ of 
$\ell^2$ has the Kadets-Klee property, which guarantees that every boundary point is a point of continuity from the weak to the norm topology. The set $K$ in Proposition~\ref{example: max isolated} is a subset of $B$, containing the basic vector $e_1:=(1,0,\dots)$ and constructed in a way that the part around $e_1$ is sufficiently flattered so that the Kadets-Klee property fails and at the same time there is no other maximal element near~$e_1$.     
\end{remark}
\bigskip

\paragraph*{\bf Acknowledgement}
A major part of this research has been realized during a research visit of the first author to the Università Cattolica del Sacro Cuore (Milano, June 2024). This author thanks his hosts for hospitality. The research of the first author was supported by the INdAM 
(GNAMPA, Professore Visitatore) and by the Austrian Science Fund (FWF) [10.55776/P36344]. The research of the second and the third author was supported by the INdAM -
GNAMPA Project,  CUP E53C23001670001, and by MICINN project PID2020-112491GB-I00 (Spain).

\end{document}